\documentclass[12pt]{article}
\usepackage{amscd}
\usepackage{amsmath}
\usepackage{latexsym}
\usepackage{amsfonts}
\usepackage{amssymb}
\usepackage{amsthm}
\usepackage{graphicx}
\usepackage{verbatim}
\usepackage{enumerate}
\usepackage{xcolor}

\theoremstyle{plain}
\theoremstyle{definition}\newtheorem{theorem}{Theorem}[section]
\theoremstyle{plain}\newtheorem{lemma}[theorem]{Lemma}
\theoremstyle{plain}
\theoremstyle{plain}
\theoremstyle{remark}\newtheorem{remark}{Remark}[section]

\newcommand{\be}{\begin{equation}}
\newcommand{\ee}{\end{equation}}
 \newcommand{\ba}{\begin{aligned}}
 \newcommand{\ea}{\end{aligned}}

  \newcommand{\ben}{\begin{enumerate}}
   \newcommand{\een}{\end{enumerate}}

\newcommand{\Rmnum}[1]{\expandafter\@slowromancap\romannumeral #1@}

\begin{document}

\begin{titlepage}
\title{ Global Regularity of 2D almost resistive MHD Equations }
\author{ Baoquan Yuan and  Jiefeng Zhao\thanks{ Corresponding author.}
\\ School of Mathematics and Information Science,
       \\ Henan Polytechnic University,  Henan,  454000,  China.\\
   (Email: bqyuan@hpu.edu.cn,\,zhaojiefeng001@163.com)}
\end{titlepage}
\date{}
\maketitle


\begin{abstract}
\noindent

    Whether or not the solution to 2D resistive MHD equations is globally smooth remains open.  This paper establishes  the global regularity of solutions to the  2D almost resistive MHD equations, which require the dissipative operators $\mathcal{L}$  weaker than any power of the fractional Laplacian. The result is an improvement of the one of Fan et al. (Global Cauchy problem of 2D generalized MHD equations, Monatsh. Math., 175 (2014), pp. 127-131) which ask for $\alpha>0, \beta=1$.
\end{abstract}

\vspace{.2in} {\bf Key words:}\quad Almost resistive MHD equations, global regularity.

\section{Introduction}

Consider the  Cauchy problem of the two-dimensional
generalized magnetohydrodynamic equations:
\begin{eqnarray}
\left\{
 \begin{array}{llll}\label{eq}
  u_t + u \cdot \nabla u  =  - \nabla p + b \cdot \nabla b - \nu \Lambda^{
  2\alpha} u,  \\
  b_t + u \cdot \nabla b = b \cdot \nabla u - \kappa \Lambda^{2\beta} b,\\
  \nabla \cdot u = \nabla \cdot b  =  0, \\
  u\left(x,0\right)=u_0\left(x\right),\,\,\, b\left(x,0\right)=b_0\left(x\right)
  \end{array}\right.
\end{eqnarray}
for $x\in \mathbb{R}^2$ and $t>0$, where $ u=u\left(x,t\right) $ is
the velocity, $ b =b\left(x,t\right) $  the magnetic,
$ p =p\left(x,t\right) $  the pressure, and $
u_0\left(x\right),\,b_0\left(x\right) $ with $\mathrm{div}
u_0\left(x\right)=\mathrm{div} b_0\left(x\right)=0$ are the initial
velocity and magnetic, respectively. Here $\nu, \kappa,
\alpha, \beta \ge 0$ are nonnegative constants and  $\Lambda=\sqrt{-\Delta}$.

The global regularity of the d-D GMHD  (\ref{eq}) has
attracted a lot of attention and progress has been made in the last few years
(see [1-8, 12-20, 22, 23, 25, 26, 28-36]).
 In 2D case, it follows from \cite{CWY2014,JiZ2015,FMMNZ2014} that the
problem (\ref{eq}) has a unique global regular solution if $\alpha = 0, \  \beta > 1$ or
   $\alpha>0,\ \beta=1$. In 2D or 3D case, there have been various results on partial regularity, Serrin type
regularity  criterions for weak solutions, or  blow-up criterions for smooth solution to the usual MHD equations, for example \cite{CKS1997,CaW2010,ChMZ2008,ChMZ-2010,HeX2005-1,HeX2005-2,LeZ2009}. Recently, some important progresses have been made
on the global well-posedness for non-resistive MHD equations ( $\kappa=0,\,\alpha=1$) near an equilibrium(see  \cite{HuL2014,LiXZ2015,LiZ2014,ReWXZ2014,XuZ2015,Zh2014}). Local existence
for 2D non-resistive MHD equations in rough spaces have been obtained in \cite{JiN2006,FMRR2014,CMRR2016,FMRR2016}. Some results on global
regularity of 2D MHD equations with partial viscosity and resistivity  refer to \cite{CaRW2013,CaW2011}. To the best of our knowledge, whether or not there exists an global
regular solution for 2D resistive MHD ($\nu=0,\,\beta=1$) is still an open problem.

In this paper, motivated by \cite{CoV2012}, we are concerned with the following 2D GMHD
\begin{eqnarray}
\left\{
 \begin{array}{llll}\label{eq2}
  u_t + u \cdot \nabla u +\mathcal{L}u =  - \nabla p + b \cdot \nabla b,  \\
  b_t + u \cdot \nabla b -\Delta b= b \cdot \nabla u ,\\
  \nabla \cdot u = \nabla \cdot b  =  0, \\
  u\left(x,0\right)=u_0\left(x\right),\,\,\,
  b\left(x,0\right)=b_0\left(x\right).
  \end{array}\right.
\end{eqnarray}
where $\mathcal{L}$ is the dissipative operator with
\begin{equation}\label{c1}
\mathcal{L}u(x)=P.V.\int_{\mathbb{R}^2} \frac{u(x)-u(x-y)}{|y|^2m(|y|)} \mathrm dy.
\end{equation}
Here $m:[0,\infty)\rightarrow [0,\infty)$ is a smooth, non-decreasing function that behaves like $\frac{1}{(-\log{r})^{1+\varepsilon_1}}$ for
   sufficiently small $r$ with  $\varepsilon_1>0$ and that grows fast at least at the rate of $(\log{r})^{1+\varepsilon_2}$ for
   sufficiently large $r$ with  $\varepsilon_2>0$, satisfying
\begin{equation}\label{c2}
\int_0^1\frac{m(r)}{r} \mathrm dr<\infty
\end{equation}
and
the doubling condition
\begin{equation}\label{c3}
m(2r)<cm(r)
\end{equation}
for some positive constants c.

The main result  of this paper is stated as follows.
\begin{theorem}\label{thm1}
   Let m(r) satisfy \eqref{c1}-\eqref{c3} and $\rho \geqslant 4$.  Assume that $u_0, b_0 \in H^\rho(\mathbb R^2)$ with $\mathrm{div}
u_0=\mathrm{div} b_0=0$. Then for any $T>0$, the Cauchy problem \eqref{eq2} has a unique   regular solution
\begin{equation*}
(u,b)\in C([0,T];H^\rho(\mathbb R^2)) \,\,\mbox{and}\,\, b\in L^2([0,T];H^{\rho+2}(\mathbb R^2)).
\end{equation*}
\end{theorem}
The existence and uniqueness are standard we omit their proofs, and only give the a priori estimates.
\begin{remark}
Due to
\begin{equation}
\Lambda^{2\alpha}u(x)=c_\alpha P.V.\int_{\mathbb{R}^2} \frac{u(x)-u(x-y)}{|y|^{2+2\alpha}} \mathrm dy
\end{equation}
for $\alpha\in(0,1)$ (see \cite{CoC2004}), the dissipative operator $\mathcal{L}$ defined in Theorem \ref{thm1} is weaker than any power of the
fractional Laplacian. Thus we improve the results in \cite{FMMNZ2014} for equations \eqref{eq} which require $\alpha>0, \beta=1$.
\end{remark}
\begin{remark}
Inspired by the work \cite{CoV2012,KN2010}, \eqref{c2} can be replaced by weaker conditions $\lim_{r\rightarrow 0^+}m(r)=0$, then we can obtain the global regularity of solutions to \eqref{eq2} with arbitrary weak dissipation $\mathcal{L}$ (see Remark \ref{rmk3}).
\end{remark}
\begin{remark}
In virtue of Remark \ref{mr1} and Section 3, we require only $u_0, b_0 \in H^\rho(\mathbb R^2)$ with $\rho > 3$.
\end{remark}
\begin{remark}
For the 2D GMHD  \eqref{eq2}, it remains an open problem whether there exists a global smooth solution without the dissipative operator $\mathcal{L}$.
\end{remark}

\section{Preliminaries}\label{sec:the preparations}
Let us first consider the heat equation
\begin{eqnarray*}
 \left\{
 \begin{array}{llll}
   v_t -\Delta v =  f,
  \\
   v\left(x,0\right)=v_0(x).
 \end{array}\right.
\end{eqnarray*}
As we all know
\begin{eqnarray}
v\left(x,t\right)&=& e^{t\Delta}v_0+\int^t_0 e^{(t-s)\Delta}f(\cdot,s)\mathrm{d}s\nonumber\\
&=&h(\cdot,t)*v_0+\int^t_0 h(\cdot,t-s)*f(\cdot,s)\mathrm{d}s,
\end{eqnarray}
where $h(x,t)=\frac{1}{(4\pi t)^{\frac{d}{2}}}e^{\frac{-|x|^2}{4t}}$.

Recalled the following maximal $L^p(L^q)$ regularity theorem for the heat kernel.

\begin{lemma}{\rm (\cite{Le2002})}\label{lem1}
Assume $f\in L^p((0,T),L^q(\mathbb{R}^d))(1<p,q<\infty)$.
Let
\begin{equation*}
 A:v\mapsto A  f(x,t)=\int^t_0  e^{(t-s)\Delta}\Delta f(\cdot,s)\mathrm{d}s,
 \end{equation*}
 then
 \begin{equation*}
 \left\|Af\right\|_{L^p((0,T),L^q(\mathbb{R}^d))}\leqslant C \left\|f\right\|_{L^p((0,T),L^q(\mathbb{R}^d))}
 \end{equation*}
for every $T\in(0,\infty]$£©and some positive constants $C$ (independent of $T$).
\end{lemma}

\begin{lemma}\label{lem2}
  Let $u_0, b_0 \in L^2(\mathbb R^2) $, for any $T>0$ and $0<t<T$, we have
  \begin{eqnarray*}
&& \left\| u \right\|_{L^2}^2 \left( t \right) + \left\| b
    \right\|_{L^2}^2 \left( t \right)\\ &&  +\frac{1}{2}\int_0^t \int_{\mathbb{R}^2}\int_{\mathbb{R}^2} \frac{\left|u(x,\tau)-u(x-y,\tau)\right|^2}{|y|^2m(|y|)} \mathrm dx\mathrm dy\mathrm{d} \tau +\int_0^t \left\| \nabla b
    \right\|_{L^2}^2  \mathrm{d} \tau \leqslant \left\| u_0 \right\|_{L^2}^2  + \left\| b_0
    \right\|_{L^2}^2  .
  \end{eqnarray*}
\end{lemma}
Due to
\begin{equation*}
\int_{\mathbb{R}^2}u\mathcal{L}u\mathrm{d}x=\frac{1}{2}\int_{\mathbb{R}^2}\int_{\mathbb{R}^2} \frac{\left|u(x,t)-u(x-y,t)\right|^2}{|y|^2m(|y|)} \mathrm dx\mathrm dy,
\end{equation*}
we get the Lemma \ref{lem2} easily by the standard $L^2$-energy estimates.

Denote $\omega = \nabla^{\bot} \cdot u = - \partial_2 u_1 +
\partial_1 u_2$ the vorticity of the velocity fields and $j = \nabla^{\bot} \cdot b = - \partial_2 b_1 + \partial_1 b_2$
the current of the magnetic fields. Applying $\nabla^{\bot} \cdot$
on both sides of the equations \eqref{eq2}, we obtain the following
equations for    $\omega$ and $j$:
  \begin{eqnarray}
    \omega_t + u \cdot \nabla \omega + \mathcal{L}\omega & = & b \cdot \nabla j,  \label{eq:omega-L2}\\
    j_t + u \cdot \nabla j-\triangle j & = & b \cdot \nabla \omega + T \left( \nabla u,
    \nabla b \right) ,  \label{eq:j-L2}
  \end{eqnarray}
 where
  \begin{equation*}
    T \left( \nabla u, \nabla b \right) = {\color{black} 2 \partial_1 b_1
    \left( \partial_1 u_2 + \partial_2 u_1 \right) + 2 \partial_2 u_2  \left(
    \partial_1 b_2 + \partial_2 b_1 \right)} .
  \end{equation*}

\begin{lemma}\label{lem3} Let $u_0, b_0 \in H^1(\mathbb R^2)$. Then
for any $T>0$ and $0<t<T$, we have
 \begin{eqnarray}
    &&\left\| \omega \right\|_{L^2}^2 \left( t \right) + \left\| j
    \right\|_{L^2}^2 \left( t \right) \\ \nonumber &&+\int_0^t \int_{\mathbb{R}^2}\int_{\mathbb{R}^2} \frac{\left|\omega(x,\tau)-\omega(x-y,\tau)\right|^2}{|y|^2m(|y|)} \mathrm dx\mathrm dy\mathrm{d} \tau +\int_0^t  \left\| \nabla j
    \right\|_{L^2}^2  \mathrm d\tau \leqslant C \left( T
    \right).
 \end{eqnarray}
 \end{lemma}
\begin{proof}
Multiplying \eqref{eq:omega-L2} by $\omega$ and \eqref{eq:j-L2} by $j$ respectively, integrating and adding together, we have
\begin{eqnarray*}
   & & \frac{1}{2}  \frac{\mathrm{d}}{\mathrm{d} t} (\left\|\omega\right\|_{L^2}^2+ \left\|j\right\|_{L^2}^2)+\frac{1}{2}\int_{\mathbb{R}^2}\int_{\mathbb{R}^2} \frac{\left|\omega(x,t)-\omega(x-y,t)\right|^2}{|y|^2m(|y|)} \mathrm dx\mathrm dy +\left\|\nabla j\right\|_{L^2}^2\\
   &=& \int_{\mathbb{R}^2} b \cdot\nabla j \, \omega \mathrm{d} x
     +\int_{\mathbb{R}^2} b\cdot\nabla \omega \, j\mathrm{d} x+ \int_{\mathbb{R}^2} T \left( \nabla u, \nabla b \right)j\mathrm{d} x\\
    &=&\int_{\mathbb{R}^2} T \left( \nabla u, \nabla b \right)j\mathrm{d} x\\
    &\leqslant&C\left\|\nabla u\right\|_{L^2}\left\|j\right\|_{L^4}^2\\
   & \leqslant& C\left\|\omega\right\|_{L^2}^2\left\|j\right\|_{L^2}^2+\frac{1}{2}\left\|\nabla j\right\|_{L^2}^2,
\end{eqnarray*}
where the Gagliardo-Nirenberg inequality has been used in the last inequality.

Thus, we have
\begin{eqnarray*}
   & &\frac{\mathrm{d}}{\mathrm{d} t} (\left\|\omega\right\|_{L^2}^2+ \left\|j\right\|_{L^2}^2)+\int_{\mathbb{R}^2}\int_{\mathbb{R}^2} \frac{\left|\omega(x,t)-\omega(x-y,t)\right|^2}{|y|^2m(|y|)} \mathrm dx\mathrm dy  +\left\|\nabla j\right\|_{L^2}^2\\
& \leqslant& C\left\|\omega\right\|_{L^2}^2\left\|j\right\|_{L^2}^2.
\end{eqnarray*}
By taking advantage of Gronwall inequality and Lemma \ref{lem2}, we complete the proof of Lemma \ref{lem3}.
\end{proof}

\begin{lemma}\label{lem4}
Let $u_0, b_0 \in H^2(\mathbb R^2)$. Then
for any $T>0$ and $0<t<T$, we have
  \begin{equation}
  b\in L^{\infty}((0,T);L^{\infty}(\mathbb{R}^2)),\,\, \nabla b\in L^{p}((0,T);L^{q}(\mathbb{R}^2)).
  \end{equation}
for any $p,q\in(2,\infty)$.
\end{lemma}
$(\ref{eq2})_2$ can be written as
\begin{equation}\label{6.4mhd}
b_t-\Delta b=\sum_{i=1}^2\partial_i (  b_i u-u_i b).
\end{equation}
Due to $b_i u-u_i b\in L^{\infty}((0,T);L^{p}(\mathbb{R}^2))$ and Lemma \ref{lem1}, we obtain Lemma \ref{lem4}.

\begin{lemma}\label{lem5}
Let $u_0, b_0 \in H^2(\mathbb R^2)$. Then
for any $T>0$ and $0<t<T$, we have
  \begin{equation}
  \omega\in L^{\infty}((0,T);L^{p}(\mathbb{R}^2))
  \end{equation}
for any $p\in(2,\infty)$.
\end{lemma}
In virtue of
\begin{equation*}
\int_{\mathbb{R}^2}|\omega|^{p-2}\omega(x)\mathcal{L}\omega(x)\mathrm{d}x\geqslant 0
\end{equation*}
for all $2\leqslant p<\infty$ (see \cite{CoV2012}), the proof of Lemma \ref{lem5} can be obtained similar to \cite{FMMNZ2014}.

\eqref{eq:j-L2} can be encoded by
\begin{equation}
j_t-\Delta j=\sum_{i=1}^2\partial_i  (b_i \omega-u_i j)+T(\nabla u,\nabla b).
\end{equation}
Similar to Lemma \ref{lem4}, we have the following lemma.
\begin{lemma}\label{lem6}
Let $u_0, b_0 \in H^2(\mathbb R^2)$. Then
for any $T>0$ and $0<t<T$, we have
  \begin{equation}
  j\in L^{\infty}((0,T);L^{r}(\mathbb{R}^2)),\,\,\nabla j\in L^{q}((0,T);L^{p}(\mathbb{R}^2))
  \end{equation}
for any $p,q\in(2,\infty)$ and $r\in(2,\infty]$.
\end{lemma}

Exploiting  the structure of the \eqref{eq2}, we can get further estimates.
\begin{lemma}\label{lem7}
Let $p,q\in[2,\infty)$, $r\in[2,\infty]$. Assume $u_0,b_0\in H^4(\mathbb{R}^2)$, then for any $T>0$, we have
\begin{eqnarray}
& &\nabla j\in L^{\infty}((0,T);L^{p}(\mathbb{R}^2)),\,\,\Delta b+b \cdot \nabla u\in L^{\infty}((0,T);L^{r}(\mathbb{R}^2)),\label{ic1}\\
& &\nabla(\Delta b+b \cdot \nabla u)\in L^{q}((0,T);L^{p}(\mathbb{R}^2))\label{ic2}.
\end{eqnarray}
\end{lemma}
\begin{remark}\label{mr1}
In fact, the estimates of \eqref{ic1} need only $u_0\in H^{2+\epsilon_3}(\mathbb{R}^2),\,b_0\in H^{2+\epsilon_3}(\mathbb{R}^2)$ with $\epsilon_3>0$.
\end{remark}
\begin{remark}
Concerning the 2D resistive MHD, we still obtain the estimates  \eqref{ic1} and \eqref{ic2}.

For a 2D Euler equation with nonlocal forces
\begin{equation*}
\omega_t+u\cdot\nabla\omega=-\partial_yu_1=\mathcal{R}_{22}\omega,
\end{equation*}
where $\mathcal{R}_{ij}\omega$ denotes the Riesz transform $\partial_{ij}\Lambda^{-2}\omega$.
Elgindi and Masmoudi \cite{ElM2014} prove that it is mildly ill-posed in $L^\infty$.

Similar to \eqref{eq:omega1}, we have
\begin{eqnarray}\label{eq:omega1}
    \omega_t + u \cdot \nabla \omega &=& b_1( \Delta b_2+b \cdot \nabla u_2)-b_2( \Delta b_1+b \cdot \nabla u_1)-b_1 b \cdot \nabla u_2+b_2 b \cdot \nabla u_1\nonumber\\
    &=& f+b_1 \sum _{i=1}^2 b_i\mathcal{R}_{i2}\omega -b_2 \sum_{i=1}^{2}b_i\mathcal{R}_{i1}\omega,
  \end{eqnarray}
  where $f=b_1( \Delta b_2+b \cdot \nabla u_2)-b_2( \Delta b_1+b \cdot \nabla u_1)\in L^\infty(0,T;L^\infty(\mathbb{R}^2))$. So the results in \cite{ElM2014}
  suggest that it may be mildly ill-posed in $L^\infty$ in the case of 2D resistive MHD.
\end{remark}


\begin{proof}
Applying $b\cdot \nabla$ and $\Delta$ to $\eqref{eq2}_1$ and $\eqref{eq2}_2$ respectively, and multiplying $\eqref{eq2}_2$  by $\nabla u$, then adding the resulting
equations together we obtain
\begin{eqnarray}
& &(\Delta b+b \cdot \nabla u)_t-\Delta(\Delta b+b \cdot \nabla u)\nonumber\\
&=&-b\cdot\nabla(u\cdot\nabla u)+b\cdot\nabla(b\cdot\nabla b)-(u\cdot\nabla b)\cdot\nabla u+(b\cdot\nabla u)\cdot\nabla u\nonumber\\
& &+\Delta b\cdot\nabla u
-b\cdot\nabla(\nabla p) -\Delta(u \cdot \nabla b)-b\cdot\nabla\mathcal{L}u.\label{equb}
\end{eqnarray}
Firstly, we give the following estimates
\begin{equation*}
\Delta b+b \cdot \nabla u\in L^{\infty}((0,T);L^{2}(\mathbb{R}^2))\bigcap L^{2}((0,T);H^{1}(\mathbb{R}^2)).
\end{equation*}
 Multiplying \eqref{equb} by $\Delta b+b \cdot \nabla u$ and integrating on $\mathbb{R}^2$, we have
\begin{eqnarray*}
&&\frac{1}{2}\frac{\mathrm{d}}{\mathrm{d} t} \left\| \Delta b+b \cdot \nabla u \right\|_{L^2}^2+
\left\|\nabla(\Delta b+b \cdot \nabla u)\right\|_{L^2}^2 \\
&=&-\int_{\mathbb{R}^2} b\cdot\nabla(u\cdot\nabla u)(\Delta b+b \cdot \nabla u)\mathrm{d}x
+\int_{\mathbb{R}^2} b\cdot\nabla(b\cdot\nabla b)(\Delta b+b \cdot \nabla u)\mathrm{d}x\\
& &-\int_{\mathbb{R}^2} (u\cdot\nabla b)\cdot\nabla u(\Delta b+b \cdot \nabla u)\mathrm{d}x
+\int_{\mathbb{R}^2} (b\cdot\nabla u)\cdot\nabla u(\Delta b+b \cdot \nabla u)\mathrm{d}x\\
& &+\int_{\mathbb{R}^2} \Delta b\cdot\nabla u(\Delta b+b \cdot \nabla u)\mathrm{d}x
-\int_{\mathbb{R}^2} b\cdot\nabla(\nabla p)(\Delta b+b \cdot \nabla u)\mathrm{d}x\\
& &-\int_{\mathbb{R}^2} \Delta(u \cdot \nabla b)(\Delta b+b \cdot \nabla u)\mathrm{d}x
-\int_{\mathbb{R}^2} b\cdot\nabla\mathcal{L}u(\Delta b+b \cdot \nabla u)\mathrm{d}x\\
&=&RHS.
\end{eqnarray*}
Thanks to
\begin{eqnarray}\label{iequb}
&&\left\|b_i\mathcal{L}u\right\|_{L^2}\nonumber\\
&=&\left(\int_{\mathbb{R}^2}\left|b_i(x)P.V.\int_{\mathbb{R}^2} \frac{u(x)-u(x-y)}{|y|^2m(|y|)}
 \mathrm dy\right|^2\mathrm{d}x\right)^{\frac{1}{2}}\nonumber\\
&=&\left(\int_{\mathbb{R}^2}\left|b_i(x)P.V.\int_{|y|\leq1} \frac{u(x)-u(x-y)}{|y|^2m(|y|)}
 \mathrm dy+b_i(x)P.V.\int_{|y|\geq1} \frac{u(x)-u(x-y)}{|y|^2m(|y|)}
 \mathrm dy\right|^2\mathrm{d}x\right)^{\frac{1}{2}}\nonumber\\
&\leqslant&C\left(\int_{\mathbb{R}^2}\left|b_i(x)\int_{|y|\leq1}\int_0^1 \frac{|(\nabla u)(x-(1-t)y)|}{|y|m(|y|)}\mathrm dt
 \mathrm dy\right|^2\mathrm{d}x\right)^{\frac{1}{2}}\nonumber\\
 &&+C\left\|b\right\|_{L^2}\left\|u\right\|_{L^\infty}\int_{|y|\geq1} \frac{1}{|y|^2m(|y|)}
 \mathrm dy\nonumber\\
&\leqslant&C (\left\|b\right\|_{L^\infty}\left\|\nabla u\right\|_{L^2}+\left\|b\right\|_{L^2}\left\|u\right\|_{L^\infty})
\end{eqnarray}
and Lemmas 4-6,
the right hand side above can be simply estimated  as follows
 \begin{eqnarray*}
RHS &\leqslant&(\left\| b\right\|_{L^6}\left\| u\right\|_{L^6}\left\| \nabla u\right\|_{L^6}+\left\| b\right\|_{L^6}^2\left\| \nabla b\right\|_{L^6})\left\|\nabla(\Delta b+b \cdot \nabla u)\right\|_{L^2}\\
& &+(\left\| u\right\|_{L^6}\left\|\nabla b\right\|_{L^6}+\left\| b\right\|_{L^6}\left\|\nabla u\right\|_{L^6})\left\|\nabla u\right\|_{L^6}\left\|\Delta b+b \cdot \nabla u\right\|_{L^2}\\
& &+\left\|\Delta b\right\|_{L^4}\left\|\nabla u\right\|_{L^4}\left\|\Delta b+b \cdot \nabla u\right\|_{L^2}
+\left\|b\right\|_{L^4}\left\|\nabla p\right\|_{L^4}\left\|\nabla(\Delta b+b \cdot \nabla u)\right\|_{L^2}\\
& &+(\left\|\nabla u\right\|_{L^4}\left\|\nabla b\right\|_{L^4}+\left\|u\right\|_{L^4}\left\|\nabla^2 b\right\|_{L^4})\left\|\nabla(\Delta b+b \cdot \nabla u)\right\|_{L^2}\\
& &+C (\left\|b\right\|_{L^\infty}\left\|\nabla u\right\|_{L^2}+\left\|b\right\|_{L^2}\left\|u\right\|_{L^\infty})\left\|\nabla(\Delta b+b \cdot \nabla u)\right\|_{L^2}\\
&\leqslant& c(t)(\left\|\Delta b+b \cdot \nabla u\right\|_{L^2}+C(T))+\frac{1}{2}\left\|\nabla(\Delta b+b \cdot \nabla u)\right\|_{L^2}
\end{eqnarray*}
where $c(t)\in L^p(0,T)(2\leqslant p<\infty)$. Taking advantage of Gronwall inequality, we get the result.

Secondly, we prove the following estimates
\begin{equation}\label{eqi}
\Delta b+b \cdot \nabla u\in L^{\infty}((0,T);L^{p}(\mathbb{R}^2))\,\,(2<p<\infty).
\end{equation}
Thus, $\Delta b\in L^{\infty}((0,T);L^{p}(\mathbb{R}^2))\,\,(2<p<\infty).$

\eqref{equb} can be written as
\begin{eqnarray}\label{6.16mhd}
(\Delta b+b \cdot \nabla u)\left(x,t\right)\triangleq I_1+I_2+I_3+I_4+I_5+I_6+I_7+I_8+I_9,
\end{eqnarray}
where\\ $I_1=h(\cdot,t)*(\Delta b_0+b_0 \cdot \nabla u_0)$, $I_2=-\int^t_0 h(\cdot,t-s)*(b\cdot\nabla(u\cdot\nabla u))(\cdot,s)\mathrm{d}s$, \\
$I_3=\int^t_0 h(\cdot,t-s)*(b\cdot\nabla(b\cdot\nabla b))(\cdot,s)\mathrm{d}s$, $I_4=-\int^t_0 h(\cdot,t-s)*((u\cdot\nabla b)\cdot\nabla u ) (\cdot,s)\mathrm{d}s$,
$I_5=\int^t_0 h(\cdot,t-s)*((b\cdot\nabla u)\cdot\nabla u)(\cdot,s)\mathrm{d}s$, $I_6=\int^t_0 h(\cdot,t-s)*(\Delta b\cdot\nabla u)(\cdot,s)\mathrm{d}s$,
$I_7=-\int^t_0 h(\cdot,t-s)*(b\cdot\nabla(\nabla p))(\cdot,s)\mathrm{d}s$, $I_8=-\int^t_0 h(\cdot,t-s)*(\Delta(u \cdot \nabla b))(\cdot,s)\mathrm{d}s$,
$I_9=-\int^t_0 h(\cdot,t-s)*( b\cdot\nabla\mathcal{L}u)(\cdot,s)\mathrm{d}s$.

Then
\begin{eqnarray}\label{esI}
\left\| I_1\right\|_{L^{\infty}((0,T);L^{p}(\mathbb{R}^2))}&=&\left\|h(\cdot,t)*(\Delta b_0+b_0 \cdot \nabla u_0)\right\|_{L^{\infty}((0,T);L^{p}(\mathbb{R}^2))}\nonumber\\
&\leqslant& C\left\| h\right\|_{L^{\infty}((0,T);L^{1}(\mathbb{R}^2))}\left\|\Delta b_0+b_0 \cdot \nabla u_0\right\|_{L^{p}(\mathbb{R}^2)}\nonumber\\
&\leqslant& C\left\|\nabla^2 b_0\right\|_{L^{p}(\mathbb{R}^2)}
+C\left\| b_0\right\|_{L^{2p}(\mathbb{R}^2)}\left\|\nabla u_0\right\|_{L^{2p}(\mathbb{R}^2)}\nonumber\\
&\leqslant&C(\left\| u_0\right\|_{H^{\rho}(\mathbb{R}^2)}+\left\| b_0\right\|_{H^{\rho}(\mathbb{R}^2)}).
\end{eqnarray}
\begin{eqnarray*}
\left\| I_2\right\|_{L^{\infty}((0,T);L^{p}(\mathbb{R}^2))}&=&\left\|\int^t_0 h(\cdot,t-s)*(b\cdot\nabla(u\cdot\nabla u))(\cdot,s)\mathrm{d}s\right\|_{L^{\infty}((0,T);L^{p}(\mathbb{R}^2))}\\
&\leqslant& C\left\|\nabla h\right\|_{L^{1}((0,T);L^{1}(\mathbb{R}^2))}\left\|bu\nabla u\right\|_{L^{\infty}((0,T);L^{p}(\mathbb{R}^2))}\\
&\leqslant& C\left\|b\right\|_{L^{\infty}((0,T);L^{\infty}(\mathbb{R}^2))}\left\|u\right\|_{L^{\infty}((0,T);L^{\infty}(\mathbb{R}^2))}
\left\|\nabla u\right\|_{L^{\infty}((0,T);L^{p}(\mathbb{R}^2))}\\
&\leqslant& C(T).
\end{eqnarray*}
Arguing similarly to above, it can be derived $\left\| I_3\right\|_{L^{\infty}((0,T);L^{p}(\mathbb{R}^2))}\leqslant C(T)$,
$\left\| I_7\right\|_{L^{\infty}((0,T);L^{p}(\mathbb{R}^2))}\leqslant C(T)$.
Using an argument deriving the estimate \eqref{iequb}, we have
$\left\|b_i\mathcal{L}u\right\|_{L^{\infty}((0,T);L^{p}(\mathbb{R}^2))}\leqslant C(T)$, so $\left\| I_9\right\|_{L^{\infty}((0,T);L^{p}(\mathbb{R}^2))}\leqslant C(T)$.

For $I_4$, we obtain
\begin{eqnarray*}
\left\| I_4\right\|_{L^{\infty}((0,T);L^{p}(\mathbb{R}^2))}&=&\left\| \int^t_0 h(\cdot,t-s)*((u\cdot\nabla b)\cdot\nabla u ) (\cdot,s)\mathrm{d}s\right\|_{L^{\infty}((0,T);L^{p}(\mathbb{R}^2))}\\
&\leqslant& C\left\| h\right\|_{L^{1}((0,T);L^{1}(\mathbb{R}^2))}\left\|(u\cdot\nabla b)\cdot\nabla u\right\|_{L^{\infty}((0,T);L^{p}(\mathbb{R}^2))}\\
&\leqslant& C\left\|u\right\|_{L^{\infty}((0,T);L^{3p}(\mathbb{R}^2))}\left\|\nabla b\right\|_{L^{\infty}((0,T);L^{3p}(\mathbb{R}^2))}
\left\|\nabla u\right\|_{L^{\infty}((0,T);L^{3p}(\mathbb{R}^2))}\\
&\leqslant& C(T).
\end{eqnarray*}
Similarly, $\left\| I_5\right\|_{L^{\infty}((0,T);L^{p}(\mathbb{R}^2))}\leqslant C(T)$.

Choosing $2<q<\infty$, one has
\begin{eqnarray*}
\left\| I_6\right\|_{L^{\infty}((0,T);L^{p}(\mathbb{R}^2))}&=&\left\| \int^t_0 h(\cdot,t-s)*(\Delta b\cdot\nabla u)(\cdot,s)\mathrm{d}s\right\|_{L^{\infty}((0,T);L^{p}(\mathbb{R}^2))}\\
&\leqslant& C\left\| h\right\|_{L^{q'}((0,T);L^{1}(\mathbb{R}^2))}\left\|\Delta b\cdot\nabla u\right\|_{L^{q}((0,T);L^{p}(\mathbb{R}^2))}\\
&\leqslant& C\left\|\Delta b\right\|_{L^{2q}((0,T);L^{2p}(\mathbb{R}^2))}
\left\|\nabla u\right\|_{L^{2q}((0,T);L^{2p}(\mathbb{R}^2))}\\
&\leqslant& C(T),
\end{eqnarray*}
\begin{eqnarray*}
\left\| I_8\right\|_{L^{\infty}((0,T);L^{p}(\mathbb{R}^2))}&=&\left\| \int^t_0 h(\cdot,t-s)*(\Delta(u \cdot \nabla b))(\cdot,s)\mathrm{d}s\right\|_{L^{\infty}((0,T);L^{p}(\mathbb{R}^2))}\\
&\leqslant& C\left\|\nabla h\right\|_{L^{q'}((0,T);L^{1}(\mathbb{R}^2))}\left\|\nabla(u\nabla b)\right\|_{L^{q}((0,T);L^{p}(\mathbb{R}^2))}\\
&\leqslant& C\left\|\nabla u\right\|_{L^{2q}((0,T);L^{2p}(\mathbb{R}^2))}
\left\|\nabla b\right\|_{L^{2q}((0,T);L^{2p}(\mathbb{R}^2))}\\
&& +C\left\|u\right\|_{L^{2q}((0,T);L^{2p}(\mathbb{R}^2))}
\left\|\nabla^2 b\right\|_{L^{2q}((0,T);L^{2p}(\mathbb{R}^2))}\\
&\leqslant& C(T),
\end{eqnarray*}
where $q$ and $q'$ satisfy $\frac{1}{q}+\frac{1}{q'}=1$ and $q'<2$. So we arrive at \eqref{eqi}.

Thirdly, we prove
\begin{equation}\label{eqi1}
\Delta b+b \cdot \nabla u\in L^{\infty}((0,T);L^{\infty}(\mathbb{R}^2)).
\end{equation}
For $I_1$, similar to \eqref{esI}, we have
\begin{equation*}
\left\| I_1\right\|_{L^{\infty}((0,T);L^{\infty}(\mathbb{R}^2))}
\leqslant C(\left\| u_0\right\|_{H^{\rho}(\mathbb{R}^2)}+\left\| b_0\right\|_{H^{\rho}(\mathbb{R}^2)}).
\end{equation*}

Let $2<p_1<\infty$ and $\frac{1}{p_1}+\frac{1}{p'_1}=1$.
\begin{eqnarray*}
\left\| I_2\right\|_{L^{\infty}((0,T);L^{\infty}(\mathbb{R}^2))}&=&\left\|\int^t_0 h(\cdot,t-s)*(b\cdot\nabla(u\cdot\nabla u))(\cdot,s)\mathrm{d}s\right\|_{L^{\infty}((0,T);L^{\infty}(\mathbb{R}^2))}\\
&\leqslant& C\left\|\nabla h\right\|_{L^{1}((0,T);L^{p'_1}(\mathbb{R}^2))}\left\|bu\nabla u\right\|_{L^{\infty}((0,T);L^{p_1}(\mathbb{R}^2))}\\
&\leqslant& C\left\|b\right\|_{L^{\infty}((0,T);L^{\infty}(\mathbb{R}^2))}\left\|u\right\|_{L^{\infty}((0,T);L^{\infty}(\mathbb{R}^2))}
\left\|\nabla u\right\|_{L^{\infty}((0,T);L^{p_1}(\mathbb{R}^2))}\\
&\leqslant& C(T).
\end{eqnarray*}
Similarly, $\left\| I_3\right\|_{L^{\infty}((0,T);L^{\infty}(\mathbb{R}^2))}\leqslant C(T)$,
$\left\| I_7\right\|_{L^{\infty}((0,T);L^{\infty}(\mathbb{R}^2))}\leqslant C(T)$,
 $\left\| I_8\right\|_{L^{\infty}((0,T);L^{\infty}(\mathbb{R}^2))}\leqslant C(T)$,
$\left\| I_9\right\|_{L^{\infty}((0,T);L^{\infty}(\mathbb{R}^2))}\leqslant C(T)$.

For $I_4$, we have
\begin{eqnarray*}
\left\| I_4\right\|_{L^{\infty}((0,T);L^{\infty}(\mathbb{R}^2))}&=&\left\| \int^t_0 h(\cdot,t-s)*((u\cdot\nabla b)\cdot\nabla u ) (\cdot,s)\mathrm{d}s\right\|_{L^{\infty}((0,T);L^{\infty}(\mathbb{R}^2))}\\
&\leqslant& C\left\| h\right\|_{L^{1}((0,T);L^{p'_1}(\mathbb{R}^2))}\left\|(u\cdot\nabla b)\cdot\nabla u\right\|_{L^{\infty}((0,T);L^{p_1}(\mathbb{R}^2))}\\
&\leqslant& C\left\|u\right\|_{L^{\infty}((0,T);L^{3p_1}(\mathbb{R}^2))}\left\|\nabla b\right\|_{L^{\infty}((0,T);L^{3p_1}(\mathbb{R}^2))}
\left\|\nabla u\right\|_{L^{\infty}((0,T);L^{3p_1}(\mathbb{R}^2))}\\
&\leqslant& C(T),
\end{eqnarray*}
Similarly, $\left\| I_5\right\|_{L^{\infty}((0,T);L^{\infty}(\mathbb{R}^2))}\leqslant C(T)$,
$\left\| I_6\right\|_{L^{\infty}((0,T);L^{\infty}(\mathbb{R}^2))}\leqslant C(T)$. And \eqref{eqi1} is proved.

Finally, we prove \eqref{ic2}.

For $\nabla I_1$, we have
\begin{eqnarray*}
\left\|\nabla I_1\right\|_{L^{q}((0,T);L^{p}(\mathbb{R}^2))}&=&\left\|\nabla(h(\cdot,t)*(\Delta b_0+b_0 \cdot \nabla u_0))\right\|_{L^{q}((0,T);L^{p}(\mathbb{R}^2))}\\
&\leqslant& C\left\| h\right\|_{L^{q}((0,T);L^{1}(\mathbb{R}^2))}\left\|\nabla(\Delta b_0+b_0 \cdot \nabla u_0)\right\|_{L^{p}(\mathbb{R}^2)}\\
&\leqslant& C\left\|\nabla^3 b_0\right\|_{L^{p}(\mathbb{R}^2)}
+C\left\|\nabla b_0\right\|_{L^{2p}(\mathbb{R}^2)}\left\|\nabla u_0\right\|_{L^{2p}(\mathbb{R}^2)}\\
& &+C\left\| b_0\right\|_{L^{2p}(\mathbb{R}^2)}\left\|\nabla^2 u_0\right\|_{L^{2p}(\mathbb{R}^2)}\leqslant C(T).
\end{eqnarray*}

Thanks to Lemma \ref{lem1}, we obtain
\begin{eqnarray*}
\left\|\nabla I_2\right\|_{L^{q}((0,T);L^{p}(\mathbb{R}^2))}&=&\left\|\nabla \int^t_0 h(\cdot,t-s)*(b\cdot\nabla(u\cdot\nabla u))(\cdot,s)\mathrm{d}s\right\|_{L^{q}((0,T);L^{p}(\mathbb{R}^2))}\\
&\leqslant& C\left\|bu\nabla u\right\|_{L^{q}((0,T);L^{p}(\mathbb{R}^2))}\\
&\leqslant& C\left\|b\right\|_{L^{\infty}((0,T);L^{\infty}(\mathbb{R}^2))}\left\|u\right\|_{L^{\infty}((0,T);L^{\infty}(\mathbb{R}^2))}
\left\|\nabla u\right\|_{L^{q}((0,T);L^{p}(\mathbb{R}^2))}\\
&\leqslant& C(T).
\end{eqnarray*}
Similarly, $\left\|\nabla I_3\right\|_{L^{q}((0,T);L^{p}(\mathbb{R}^2))}\leqslant C(T)$,
$\left\|\nabla I_7\right\|_{L^{q}((0,T);L^{p}(\mathbb{R}^2))}\leqslant C(T)$,
$\left\|\nabla I_8\right\|_{L^{q}((0,T);L^{p}(\mathbb{R}^2))}\leqslant C(T)$,
$\left\|\nabla I_9\right\|_{L^{q}((0,T);L^{p}(\mathbb{R}^2))}\leqslant C(T)$.

For $\nabla I_4$, we get
\begin{eqnarray*}
\left\|\nabla I_4\right\|_{L^{q}((0,T);L^{p}(\mathbb{R}^2))}&=&\left\|\nabla \int^t_0 h(\cdot,t-s)*((u\cdot\nabla b)\cdot\nabla u ) (\cdot,s)\mathrm{d}s\right\|_{L^{q}((0,T);L^{p}(\mathbb{R}^2))}\\
&\leqslant& C\left\|\nabla h\right\|_{L^{1}((0,T);L^{1}(\mathbb{R}^2))}\left\|(u\cdot\nabla b)\cdot\nabla u\right\|_{L^{q}((0,T);L^{p}(\mathbb{R}^2))}\\
&\leqslant& C\left\|u\right\|_{L^{3q}((0,T);L^{3p}(\mathbb{R}^2))}\left\|\nabla b\right\|_{L^{3q}((0,T);L^{3p}(\mathbb{R}^2))}
\left\|\nabla u\right\|_{L^{3q}((0,T);L^{3p}(\mathbb{R}^2))}\\
&\leqslant& C(T).
\end{eqnarray*}
Similarly, $\left\|\nabla I_4\right\|_{L^{q}((0,T);L^{p}(\mathbb{R}^2))}\leqslant C(T)$,
$\left\|\nabla I_5\right\|_{L^{q}((0,T);L^{p}(\mathbb{R}^2))}\leqslant C(T)$,
$\left\|\nabla I_6\right\|_{L^{q}((0,T);L^{p}(\mathbb{R}^2))}\leqslant C(T)$.

Therefore, we obtain \eqref{ic2} and finish the proof of lemma \ref{lem7}.
\end{proof}

\section{The Proof of Theorem 1.1}\label{sec:the proof of theorem 1.1}

Due to
\begin{equation*}
\partial_1 j=\Delta b_2,\,\,\,\partial_2 j=-\Delta b_1,
\end{equation*}
(\ref{eq:omega-L2}) can be changed into
\begin{eqnarray}\label{eq:omega1}
    \omega_t + u \cdot \nabla \omega +\mathcal{L}\omega&=& b_1( \Delta b_2+b \cdot \nabla u_2)-b_2( \Delta b_1+b \cdot \nabla u_1)-b_1 b \cdot \nabla u_2+b_2 b \cdot \nabla u_1\nonumber\\
    &=& f-b_1 b \cdot \nabla u_2+b_2 b \cdot \nabla u_1,
  \end{eqnarray}
  where $f=b_1( \Delta b_2+b \cdot \nabla u_2)-b_2( \Delta b_1+b \cdot \nabla u_1)$.

Multiplying \eqref{eq:omega1} by $\omega(x,t)$, we obtain
\begin{equation*}
\frac{1}{2}(\partial_t+u\cdot\nabla)\left|\omega(x,t)\right|^2+\omega(x,t)\mathcal{L}\omega(x,t)=
(f-b_1 b \cdot \nabla u_2+b_2 b \cdot \nabla u_1)(x,t)\omega(x,t).
\end{equation*}
Using the pointwise identity
\begin{equation*}
\omega(x,t)\mathcal{L}\omega(x,t)=\frac{1}{2}\mathcal{L}(\left|\omega(x,t)\right|^2)+\frac{D(x,t)}{2}
\end{equation*}
(see \cite{CoV2012}), where
\begin{equation*}
D(x,t)=P.V.\int_{\mathbb{R}^2} \frac{(\omega(x,t)-\omega(x-y,t))^2}{|y|^2m(|y|)} \mathrm dy,
\end{equation*}
we get
\begin{equation}\label{eq:omega2}
\frac{1}{2}(\partial_t+u\cdot\nabla+\mathcal{L})\left|\omega(x,t)\right|^2+\frac{D(x,t)}{2}=
(f-b_1 b \cdot \nabla u_2+b_2 b \cdot \nabla u_1)(x,t)\omega(x,t).
\end{equation}
Choosing a non-negative radial smooth cut-off function $\chi_1(x)$ supported in $|x|\leqslant1$, identically equal to 1
for $|x|\leqslant\frac{1}{2}$ and $|\nabla\chi_1(x)|\leqslant C$. Let $\chi_2(x)=1-\chi_1(x)$.

By Biot-Savart law \cite{MaB2002},
\begin{equation}\label{bse}
u(x,t)=\frac{1}{2\pi}\int_{\mathbb{R}^2}(-\frac{y_2}{|y|^2},\frac{y_1}{|y|^2})\omega(x-y,t)\mathrm dy,
\end{equation}
so
\begin{eqnarray}\label{ieqomega}
&&\left|(b_1b\cdot\nabla u_2\omega)(x,t)\right|\nonumber\\
&=&\left|b_1(x,t)\omega(x,t)b(x,t)\cdot \frac{1}{2\pi}\int_{\mathbb{R}^2}\frac{y_1}{|y|^2}\nabla_y\omega(x-y,t)\mathrm dy\right|\nonumber\\
&\leqslant&\left|b_1(x,t)\omega(x,t)b(x,t)\cdot \frac{1}{2\pi}\int_{|y|\leq1}\frac{y_1}{|y|^2}\nabla_y(\omega(x,t)-\omega(x-y,t))\chi_1(y)\mathrm dy\right|\nonumber\\
&+&\left|b_1(x,t)\omega(x,t)b(x,t)\cdot \frac{1}{2\pi}\int_{|y|\geq\frac{1}{2}}\frac{y_1}{|y|^2}\nabla_y\omega(x-y,t)\chi_2(y)\mathrm dy\right|\nonumber\\
&\leqslant&c_1 \left\|b\right\|_{L^\infty}^2|\omega(x,t)|\int_{|y|\leq1}\frac{1}{|y|^2}|\omega(x,t)-\omega(x-y,t)|\mathrm dy\nonumber\\
&+&c_2 \left\|b\right\|_{L^\infty}^2|\omega(x,t)|\int_{|y|\geq\frac{1}{2}}\frac{1}{|y|^2}|\omega(x-y,t)|\mathrm dy\nonumber\\
&\leqslant&c_1 \left\|b\right\|_{L^\infty}^2|\omega(x,t)|\int_{|y|\leq1}\frac{|\omega(x,t)-\omega(x-y,t)|}{|y|\sqrt{m(|y|)}}
\frac{\sqrt{m(|y|)}}{|y|}\mathrm dy+c_3\left\|b\right\|_{L^\infty}^2\left\|\omega\right\|_{L^2}|\omega(x,t)|\nonumber\\
&\leqslant&c_4 \left\|b\right\|_{L^\infty}^2|\omega(x,t)|\sqrt{D(x,t)}(\int_0^1\frac{m(r)}{r} \mathrm dr)^{\frac{1}{2}}+c_3\left\|b\right\|_{L^\infty}^2\left\|\omega\right\|_{L^2}|\omega(x,t)|\nonumber\\
&\leqslant &\frac{D(x,t)}{8}+c_5\left\|b\right\|_{L^\infty}^4|\omega(x,t)|^2+c_3\left\|b\right\|_{L^\infty}^2\left\|\omega\right\|_{L^2}|\omega(x,t)|.
\end{eqnarray}
Similarly,
\begin{eqnarray*}
&&\left|(b_2b\cdot\nabla u_1\omega)(x,t)\right|\nonumber\\
&\leqslant&\frac{D(x,t)}{8}+c_5\left\|b\right\|_{L^\infty}^4|\omega(x,t)|^2+c_3\left\|b\right\|_{L^\infty}^2\left\|\omega\right\|_{L^2}|\omega(x,t)|.
\end{eqnarray*}
Thus, thanks to Lemma \ref{lem7}, \eqref{eq:omega2} and \eqref{ieqomega} give
\begin{eqnarray}
&&\frac{1}{2}(\partial_t+u\cdot\nabla+\mathcal{L})\left|\omega(x,t)\right|^2+\frac{D(x,t)}{4}\nonumber\\
&\leqslant&2c_5\left\|b\right\|_{L^\infty}^4|\omega(x,t)|^2+
(2c_3\left\|b\right\|_{L^\infty}^2\left\|\omega\right\|_{L^2}+\left\|f\right\|_{L^\infty})|\omega(x,t)|\nonumber\\
&\leqslant&2c_5\left\|b\right\|_{L^\infty}^4|\omega(x,t)|^2+
(2c_3\left\|b\right\|_{L^\infty}^2\left\|\omega\right\|_{L^2}+2\left\|b\right\|_{L^\infty}\left\|\Delta b+b \cdot \nabla u\right\|_{L^\infty})|\omega(x,t)|\nonumber\\
&\leqslant& C_1(T)(|\omega(x,t)|+|\omega(x,t)|^2).
\end{eqnarray}
Since
\begin{equation*}
D(x,t)\geqslant \frac{c_6}{m(1)}|\omega(x,t)|^2\log\frac{1}{\delta}-c_7|\omega(x,t)|\left\|\omega\right\|_{L^2}\frac{1}{\delta m(\delta)}
\end{equation*}
(see (5.18) in \cite{CoV2012}), where $\delta<1$. We pick $\delta=\delta(m,T)\in(0,1)$ to be such that
\begin{equation*}
\frac{c_6}{8m(1)}\log\frac{1}{\delta}>C_1(T).
\end{equation*}
Hence,
\begin{eqnarray}\label{ieqomega1}
&&\frac{1}{2}(\partial_t+u\cdot\nabla+\mathcal{L})\left|\omega(x,t)\right|^2+C_2(T)|\omega(x,t)|^2\nonumber\\
&\leqslant&(C_1(T)+C_3(T)\left\|\omega\right\|_{L^2})|\omega(x,t)|\nonumber\\
&\leqslant&C_4(T)|\omega(x,t)|.
\end{eqnarray}
Let $\varphi(r)$ be a non-decreasing positive convex smooth function which is identically 0 on
 $0\leqslant r\leqslant \mathrm{max}\{\left\|\omega_0\right\|_{L^\infty}^2,(\frac{C_4(T)}{C_1(T)})^2\}$.
 Multiplying \eqref{ieqomega1} by $\varphi^{'}(|\omega(x,t)|^2)$ gives
 \begin{equation}\label{ieqomega2}
 (\partial_t+u\cdot\nabla+\mathcal{L})\varphi(|\omega(x,t)|^2)\leqslant 0
 \end{equation}
 for all $x$ and all $t\in[0,T)$. Thanks to
 \begin{equation*}
\int_{\mathbb{R}^2}|\omega(x)|^{p-2}\omega(x)\mathcal{L}\omega(x)\mathrm{d}x\geqslant 0,
\end{equation*}
for all $2\leqslant p<\infty$ (see \cite{CoV2012}). Hence from \eqref{ieqomega2}, we obtain
\begin{equation*}
\left\|\varphi(|\omega(x,t)|^2)\right\|_{L^\infty}\leqslant\left\|\varphi(|\omega_0|^2)\right\|_{L^\infty}=0.
\end{equation*}
This gives that $\left\|\omega(\cdot,t)\right\|_{L^\infty}\leqslant \mathrm{max}\{\left\|\omega_0\right\|_{L^\infty},\frac{C_4(T)}{C_1(T)}\}$
for all $t\in [0,T)$.

By taking advantage of the BKM type criterion for global regularity (see \cite{CKS1997}), we finish the proof of Theorem \ref{thm1}.
\begin{remark}\label{rmk3}
Similar to Remark 5.4 in \cite{CoV2012}, if \eqref{c2} is replaced by  $\lim_{r\rightarrow 0^+}m(r)=0$,  we can still obtain the global regularity of \eqref{eq2}.
We only give $\omega\in L^\infty([0,T);L^\infty(\mathbb{R}^2))$. A sketch of the proof is as follows. We can assume that $\sup_{x\in\mathbb{R}^2}\omega(x,t)$ is obtained at $\bar{x}(t)$ for $t\in[0.T)$, if not, we only need to consider  \eqref{eq:omega1} multiplied by a smooth cut-off function. Then, at $\bar{x}$ the convection term in \eqref{eq:omega1}
vanishes and we have
\begin{equation*}
    \partial_t\omega(\bar{x},t)+ \mathcal{L}\omega(\bar{x},t)= ( f-b_1 b \cdot \nabla u_2+b_2 b \cdot \nabla u_1)(\bar{x},t).
\end{equation*}
Choosing a non-negative radial smooth cut-off function $\chi_3(x)$ supported in $|x|\leqslant \eta\,(\eta>0)$, identically equal to 1
for $|x|\leqslant\frac{1}{2}\eta$ and $|\nabla\chi_3(x)|\leqslant \frac{C}{\eta}$. Let $\chi_4(x)=1-\chi_3(x)$. Then by \eqref{bse}, we have
\begin{eqnarray*}
&&\left|(b_1b\cdot\nabla u_2)(\bar{x},t)\right|\nonumber\\
&=&\left|b_1(\bar{x},t)b(\bar{x},t)\cdot \frac{1}{2\pi}\int_{\mathbb{R}^2}\frac{y_1}{|y|^2}\nabla_y\omega(\bar{x}-y,t)\mathrm dy\right|\nonumber\\
&\leqslant&\left|b_1(\bar{x},t)b(\bar{x},t)\cdot \frac{1}{2\pi}\int_{|y|\leqslant\eta}\frac{y_1}{|y|^2}\nabla_y(\omega(\bar{x},t)-\omega(\bar{x}-y,t))\chi_3(y)\mathrm dy\right|\nonumber\\
&+&\left|b_1(\bar{x},t)b(\bar{x},t)\cdot \frac{1}{2\pi}\int_{|y|\geqslant\frac{1}{2}\eta}\frac{y_1}{|y|^2}\nabla_y\omega(\bar{x}-y,t)\chi_4(y)\mathrm dy\right|\nonumber\\
&\leqslant&C \left\|b\right\|_{L^\infty}^2\int_{|y|\leqslant\eta}\frac{1}{|y|^2}|\omega(\bar{x},t)-\omega(\bar{x}-y,t)|\mathrm dy\nonumber\\
&+&C\left\|b\right\|_{L^\infty}^2\int_{|y|\geqslant\frac{1}{2}\eta}\frac{1}{|y|^2}|\omega(\bar{x}-y,t)|\mathrm dy.
\end{eqnarray*}
Similarly, we have
\begin{eqnarray*}
&&\left|(b_1b\cdot\nabla u_2)(\bar{x},t)\right|\nonumber\\
&\leqslant&C \left\|b\right\|_{L^\infty}^2\int_{|y|\leqslant\eta}\frac{1}{|y|^2}|\omega(\bar{x},t)-\omega(\bar{x}-y,t)|\mathrm dy\nonumber\\
&+&C \left\|b\right\|_{L^\infty}^2\int_{|y|\geqslant\frac{1}{2}\eta}\frac{1}{|y|^2}|\omega(\bar{x}-y,t)|\mathrm dy.
\end{eqnarray*}
Hence,
\begin{eqnarray*}
\partial_t\omega(\bar{x},t)&\leqslant&\int_{\mathbb{R}^2} \frac{\omega(\bar{x}-y,t)-\omega(\bar{x},t)}{|y|^2m(|y|)} \mathrm dy+C \left\|b\right\|_{L^\infty}^2\int_{|y|\leqslant\eta}\frac{1}{|y|^2}|\omega(\bar{x},t)-\omega(\bar{x}-y,t)|\mathrm dy\nonumber\\
&+&C \left\|b\right\|_{L^\infty}^2\int_{|y|\geqslant\frac{1}{2}\eta}\frac{1}{|y|^2}|\omega(\bar{x}-y,t)|\mathrm dy+f(\bar{x},t)\\
&\leqslant&\int_{|y|\leqslant\eta}\frac{\omega(\bar{x}-y,t)-\omega(\bar{x},t)}{|y|^2}
\left(\frac{1}{m(|y|)}-C\left\|b\right\|_{L^\infty}^2\right)\mathrm dy\\
&&-\omega(\bar{x},t)\int_{|y|\geqslant\eta} \frac{1}{|y|^2m(|y|)}\mathrm dy
+C \frac{1}{\eta}(\left\|b\right\|_{L^\infty}^2+\frac{1}{m(\eta)})\left\|\omega\right\|_{L^2}+\left\|f\right\|_{L^\infty}.
\end{eqnarray*}
Thanks to Lemma \ref{lem4}, we can choose $\eta$ (dependent on $T$) small so that $\frac{1}{m(|y|)}-C\left\|b\right\|_{L^\infty}^2>0$. Due to Lemma \ref{lem7}, we obtain
\begin{equation*}
\partial_t\omega(\bar{x},t)\leqslant C_5(T)-C_6(T)\omega(\bar{x},t).
\end{equation*}
Therefore $\omega(\bar{x},t)\leqslant C(T)$ for $t\in[0,T)$. A similar argument can be applied to the minimum  and we obtain $\left\|\omega(\cdot,t)\right\|_{L^\infty}\leqslant C(T) $ for all $t\in [0,T).$
\end{remark}


\textbf{Acknowledgement} The research of B Yuan
was partially supported by the National Natural Science Foundation
of China (No. 11471103).



\begin{thebibliography}{10}

    \bibitem{CKS1997}R.~E.~Caflisch, I.~Klapper, G.~Steele,{\newblock} Remarks
    on singularities, dimension and energy dissipation for ideal hydrodynamics
    and MHD,{\newblock} Comm. Math. Phys. 184 (1997) 443--455.{\newblock}

    \bibitem{CMRR2016}J.~Y.~Chemin, D.~S.~McCormick, J.~C.~Robinson, J.~L.~Rodrigo,{\newblock}
    Local existence for the non-resistive MHD equations in Besov spaces,{\newblock}
    Adv. Math. 286 (2016) 1--31{\newblock}

    \bibitem{CaRW2013}C.~Cao, D.~Regmi, J.~Wu,{\newblock} The 2D MHD equations with horizontal
     dissipation and horizontal magnetic diffusion,{\newblock}
     J. Differential Equations 254 (7) (2013) 2661--2681. {\newblock}

     \bibitem{CaW2011}C.~Cao, J.~Wu,{\newblock} Global regularity for the 2D MHD equations with
     mixed partial dissipation and magnetic diffusion,{\newblock}
      Adv. Math. 226 (2) (2011) 1803--1822. {\newblock}


     \bibitem{CaW2010}C.~Cao, J.~Wu,{\newblock} Two regularity criteria for the 3D MHD equations,{\newblock}
      J. Differential Equations 248 (9) (2010) 2263--2274.{\newblock}

    \bibitem{CWY2014}C.~Cao, J.~Wu, B.~Yuan,{\newblock} The 2D incompressible
    magnetohydrodynamics equations with only magnteic diffusion,{\newblock}
     SIAM J. Math.ANAL. 46 (1) (2014) 588--602.{\newblock}


    \bibitem{ChMZ2008}Q.~Chen, C.~Miao, Z.~Zhang,{\newblock}  On the regularity criterion of
    weak solution for the 3D viscous magneto-hydrodynamics equations,{\newblock}
     Comm. Math. Phys. 284 (3) (2008)  919--930.{\newblock}

   \bibitem{ChMZ-2010}Q.~Chen, C.~Miao, Z.~Zhang,{\newblock} On the well-posedness
     of the ideal MHD equations in the Triebel-Lizorkin spaces,{\newblock}
     Arch. Ration. Mech. Anal. 195 (2) (2010) 561--578.{\newblock}

    \bibitem{CoV2012}P.~Constantin, V.~Vicol,{\newblock} Nonlinear maximum principles for
   dissipative linear nonlocal operators and applications,{\newblock} Geom. Funct. Anal.
   22 (5) (2012) 1289--1321.{\newblock}

   \bibitem{CoC2004}A.~C\'{o}rdoba, D.~C\'{o}rdoba,{\newblock} A maximum principle applied to
    quasi-geostrophic equations,{\newblock}  Comm. Math. Phys. 249 (3) (2004) 511--528.{\newblock}

    \bibitem{ElM2014}T.~M.~Elgindi, N.~Masmoudi, $L^\infty$ Ill-posedness for a class of equations
    arising in hydrodynamics, arXiv:1405.2478v2[math.AP]24 Jun 2014.

   \bibitem{FMMNZ2014}J.~Fan, H.~Malaikah, S.~Monaquel, G.~Nakamura, Y.~Zhou,{\newblock}
    Global Cauchy problem of 2D generalized MHD equations,{\newblock}  Monatsh. Math. 175
    (2014) 127--131.{\newblock}

    \bibitem{FMRR2014}C.~L.~Fefferman. D.~S.~McCormick, J.~C.~Robinson. J.~L.~Rodrigo,{\newblock}
     Higher order commutator estimates and local existence for the non-resistive MHD equations
     and related models,{\newblock} J. Funct. Anal. 267 (2014) 1035--1056{\newblock}

     \bibitem{FMRR2016}C.~L.~Fefferman. D.~S.~McCormick, J.~C.~Robinson. J.~L.~Rodrigo,{\newblock}
     Local existence for the non-resistive MHD equations in nearly optimal Sobolev spaces,{\newblock}
     arXiv: 1602.02588v1[math.AP] 2 Feb 2015.{\newblock}

    \bibitem{HeX2005-1}C.~He, Z.~Xin,{\newblock} On the regularity of weak solutions to the
    magnetohydrodynamic equations,{\newblock} J. Differential Equations 213 (2) (2005) 234--254.{\newblock}


    \bibitem{HeX2005-2}C.~He, Z.~Xin,{\newblock} Partial regularity of suitable weak solutions to the
    incompressible magnetohydrodynamic equations,{\newblock} J. Funct. Anal. 227 (2005) 113--152.{\newblock}

    \bibitem{HuL2014}X.~Hu, F.~Lin,{\newblock} Global existence for two dimensional incompressible magnetohydrodynamic
    flows with zero magnetic diffusivity,{\newblock} arXiv: 1405.0082v1[math.AP] 1 May 2014.{\newblock}

    \bibitem{JiN2006}Q.~Jiu, D.~Niu,{\newblock} Mathematical results related to a two-dimensional
   magneto-hydrodynamic equations, {\newblock} Acta Math. Sci. Ser. B English. Ed. 26
    (2006) 744--756.{\newblock}

   \bibitem{JiZ2014}Q.~Jiu, J.~Zhao,{\newblock}  A remark on global regularity of
   2D generalized magnetohydrodynamic equations,{\newblock} J. Math. Anal. Appl. 412 (2014)
   478--484.{\newblock}

   \bibitem{JiZ2015}Q.~Jiu, J.~Zhao,{\newblock}  Global regularity of 2D generalized MHD
   equations with magnetic diffusion,{\newblock} Z. Angew. Math. Phys. 66 (2015) 677--687{\newblock}

   \bibitem{KN2010}A.~Kiselev, F,~Nazarov,{\newblock} Global regularity for the critical dispersive
   dissipative surface quasi-geostrophic equation,{\newblock} Nonlinearity 23 (2010) 549--554{\newblock}

   \bibitem{Le2015}Z.~Lei,{\newblock} On axially symmetric incompressible magnetohydrodynamics in three
   dimensions,{\newblock} J. Differential Equations. 259 (2015) 3202--3215{\newblock}

   \bibitem{LeZ2009}Z.~Lei, Y.~Zhou,{\newblock}  BKM's Criterion and global weak solutions for magnetohydrodynamics
   with zero viscosity,{\newblock} Discrete Contin. Dyn. Syst. 25 (2009) 575--583.{\newblock}

   \bibitem{Le2002}P.~G.~Lemari\'{e}-Rieusset,{\newblock} Recent developments in the Navier-Stokes problem,
      {\newblock}Chapman $\&$ Hall/CRC Research Notes in Mathematics 431,Chapman $\&$ Hall/CRC, Boca Raton, FL (2002). {\newblock}

   \bibitem{LiXZ2015}F.~Lin, L.~Xu, P.~Zhang,{\newblock} Global small solutions of 2-D incompressible
   MHD system,{\newblock} J. Differential Equations. 259 (2015) 5440--5485{\newblock}

   \bibitem{LiZ2014}F.~Lin, P.~Zhang,{\newblock} Global small solutions to an MHD-type system: the three-dimensional
   case,{\newblock} Comm. Pure Appl. Math. 67 (2014) 531--580.{\newblock}

   \bibitem{MaB2002}A.~Majda, A.~Bertozzi,{\newblock} Vorticity and incompressible flow,{\newblock} Cambridge University
   Press (2002).{\newblock}


   \bibitem{ReWXZ2014}X.~Ren, J.~ Wu, Z.~Xiang, Z.~Zhang,{\newblock} Global existence and decay of smooth
   solution for the 2-D MHD equations without magnetic diffusion,{\newblock}
   J. Funct. Anal. 267 (2) (2014) 503--541.{\newblock}

   \bibitem{SeT1983} S.~Sermange, R.~Temam,{\newblock} Some mathematical questions related to the
    MHD equations,{\newblock} Comm. Pure Appl. Math. 36 (1983) 635--664.{\newblock}

  \bibitem{TYZ2013}C.V.~Tran, X.~Yu, Z.~Zhai,{\newblock} On global
  regularity of 2D generalized magnetodydrodynamics equations,{\newblock}
  J. Differential Equations. 254 (2013) 4194--4216. {\newblock}


  \bibitem{W2003}J.~Wu,{\newblock} Generalized MHD
  equations,{\newblock} J. Differential Equations. 195 (2) (2003) 284--312{\newblock}

  \bibitem{W2011}J.~Wu,{\newblock} Global regularity for a class of
  generalized magnetohydrodynamic equations,{\newblock} J. Math.
  Fluid Mech. 13 (2011) 295--305. {\newblock}

  \bibitem{XuZ2015}L.~Xu, P.~Zhang,{\newblock} Global small solutions to three-dimensional
  incompressible MHD system,{\newblock} SIAM J.Math Anal. 47 (2015) 26--65
  {\newblock}

  \bibitem{Y2014}K.~Yamazaki,{\newblock} On the global regularity of two-dimensional
  generalized magnetohydrodynamics system,{\newblock} J. Math. Anal. Appl. 416 (2014)
    99--111.{\newblock}

  \bibitem{YuB2014}B.~Yuan, L.~Bai,{\newblock} Remarks on global regularity of
  2D generalized MHD equations,{\newblock} J. Math. Anal. Appl. 413 (2014)
    633--640.{\newblock}


   \bibitem{Zh2014}T.~Zhang,{\newblock} An elementary proof of the global existence and
   uniqueness theorem to 2D incompressible non-resistive MHD system,{\newblock}
    arXiv: 1404.5681v2[math.AP] 23 Oct 2014.{\newblock}


\end{thebibliography}
\end{document}